\documentclass{smfart}
\usepackage{smfenum,smfcs,amssymb,mathrsfs,upref}
\let\mathcal\mathscr
\usepackage[latin1]{inputenc}
\def\resp{\text{resp.}\kern.3em}
\let\CC\bbC
\let\NN\bbN
\let\PP\bbP
\let\RR\bbR
\let\QQ\bbQ
\let\ZZ\bbZ

\def\cf{\textit{cf.}\kern.3em}
\def\eg{\textit{e.g.}\kern.3em}
\def\ie{\textit{i.e.}\kern.3em}
\def\resp{resp.\kern.3em}

\let\leq\leqslant
\let\geq\geqslant
\let\epsilon\varepsilon
\let\<\langle
\let\>\rangle
\let\oldS\S
\def\S{\oldS\kern .15em}
\newcommand{\cbbullet}{{\raisebox{1pt}{$\bbullet$}}}
\newcommand{\itwopi}{\tfrac{i}{2\pi}}
\newcommand{\rmod}{\mathrm{mod}}
\newcommand{\rN}{\textup{N}}
\newcommand{\rNM}{\textup{NM}}

\DeclareMathOperator{\CRH}{C^\infty RH}
\DeclareMathOperator{\CRHm}{C^\infty RH^\rmod}
\newcommand{\wT}{\mathsf{T}}
\def\thesubsection{\arabic{subsection}}
\NumberTheoremsAs{equation}
\newenvironment{enumeratea}
{\bgroup\def\theenumi{\alph{enumi}}\begin{enumerate}}
{\end{enumerate}\egroup}

\begin{document}
\frontmatter
\title[Une remarque sur \og Thom-Sebastiani et intégrales-fibres\fg]{Une remarque sur l'article\\\og Un théorème à la ``Thom-Sebastiani'' pour les~intégrales-fibres\fg de D.~Barlet}
\author[C.~Sabbah]{Claude Sabbah}
\address{UMR 7640 du CNRS\\
Centre de Mathématiques Laurent Schwartz\\
\'Ecole polytechnique\\
F--91128 Palaiseau cedex\\
France}
\email{sabbah@math.polytechnique.fr}
\urladdr{http://www.math.polytechnique.fr/~sabbah}

\thanks{Cette recherche a été partiellement financée par le programme ANR-08-BLAN-0317-01 de l'Agence nationale de la recherche.}

\begin{abstract}
Dans cette note, nous donnons une autre démonstration du résultat démontré par D\ptbl Barlet dans \cite{Barlet08}, en nous appuyant sur la transformation de Mellin et des propriétés classiques des fonctions de Bessel (au lieu de la convolution). Le théorème principal de cette note est un peu plus précis que l'énoncé donné dans \cite{Barlet08}.
\end{abstract}

\alttitle{A remark on the article \og Un théorème à la ``Thom-Sebastiani'' pour les~intégrales-fibres\fg\ by D.~Barlet}
\begin{altabstract}
In this article, we give another proof of the result shown by D\ptbl Barlet in \cite{Barlet08}, relying on the Mellin transform and on classical properties of Bessel functions (instead of convolution). The main theorem of this note is somewhat more precise than the statement given in \cite{Barlet08}.
\end{altabstract}



\altkeywords{Thom-Sebastiani, Mellin transformation, Bessel function}

\maketitle
\mainmatter

\subsection{Introduction}
Soient $f:(\CC^m,0)\to(\CC,0)$ et $g:(\CC^n,0)\to(\CC,0)$ deux germes de fonctions holomorphes, et $f\oplus g:(\CC^{m+n},0)\to(\CC,0)$ leur somme de Thom-Sebastiani, définie par $(f\oplus g)(x,y)=f(x)+g(y)$. Si $f$ (\resp $g$) est définie sur un ouvert~$U$ (\resp $V$), fixons une forme $\varphi$ (\resp $\psi$) de degré maximum sur $U$ (\resp $V$), qui est $C^\infty$ à support compact et considérons la distribution à support compact sur $\CC$ qui, à toute fonction $C^\infty$ $\chi(s)$ associe
\[
\big\langle T_{f,\varphi},\chi\big\rangle=\frac{i}{2\pi}\int_{\CC^m}(\chi\circ f)\cdot\varphi=\int_\CC\Big(\int_{f=s}\frac{\varphi}{df\wedge d\ov f}\Big)\chi(s)\itwopi\,ds\wedge d\ov s,
\]
et idem pour $g$. C'est en particulier une distribution tempérée sur $\CC$.

Il est connu (\cf \cite{Barlet82}) que $T_{f,\varphi}$ peut s'écrire sous la forme d'une somme finie
\[
T_{f,\varphi}=\sum_{r\in{}]-1,0]}\sum_{k\geq0}\theta_{r,k}(s)|s|^{2r}(\log|s|)^k,
\]
avec $\theta_{r,k}(s)$ $C^\infty$ au voisinage de $s=0$. On s'intéressera seulement à la partie singulière de ce développement, provoquée par les singularités de $f$, c'est-à-dire qu'on travaillera modulo les germes de fonctions $C^\infty$ en $s=0$. Le problème de Thom-Sebastiani dans ce cadre consiste à exprimer le développement de $T_{f\oplus g,\varphi\wedge\psi}$ (modulo fonctions $C^\infty$) en fonction de ceux de $T_{f,\varphi}$ et $T_{g,\psi}$. Il est clair qu'une formule complète ne peut être que compliquée, mais on cherche à récupérer les termes dominants de ce développement. On introduit donc la notion de polygone de Newton, qu'on note $\rN(T/C^\infty)$, en considérant le développement ci-dessus comme un développement en série de deux variables $s$ et $\ov s$ à coefficients dans l'anneau des polynômes en $\log|s|$. Nous allons montrer la relation (où la somme des polygones est la somme de Minkowski):
\[
\rN(T_{f\oplus g,\varphi\wedge\psi}/C^\infty)=\rN(T_{f,\varphi}/C^\infty)+\rN(T_{g,\psi}/C^\infty)+(1,1).
\]

Le résultat est cependant plus précis. À tout sommet $(m'+r,m''+r)$ de $\rN(T/C^\infty)$, nous attachons un nombre complexe non nul et un entier, sous la forme d'un monôme $\wh\bC_{r,m',m''}u^\ell$ et obtenons ainsi un polygone de Newton décoré $\wh\rN(T/C^\infty)$. Le résultat principal est alors:

\begin{theoreme*}[de type \og Thom-Sebastiani\fg]
On a la relation entre polygones de Newton décorés:
\[
\wh \rN(T_{f\oplus g,\varphi\wedge\psi}/C^\infty)=\wh \rN(T_{f,\varphi}/C^\infty)+\wh \rN(T_{g,\psi}/C^\infty)+(1,1).
\]
\end{theoreme*}

Expliquons la signification de la somme de Minkowski dans ce cadre.

\begin{lemme*}
\'Etant donnés deux convexes $\rN_1$ et $\rN_2$ dans $\RR^2$ qui sont chacun enveloppe convexe d'un nombre fini de quadrants $m+(\RR_+)^2$ ($m\in\RR^2$), la somme de Minkowski $\rN_1+\rN_2$ est du même type et, pour chaque sommet $m$ de $\rN_1+\rN_2$, l'écriture $m=m_1+m_2$ avec $m_1\in \rN_1$ et $m_2\in \rN_2$ est unique, et $m_i$ est un sommet de $\rN_i$ ($i=1,2$).
\end{lemme*}

\begin{proof}
On considère deux ensembles finis $\Sigma_1,\Sigma_2\subset\RR^2$, tels que $\rN_i=\mathrm{conv}\,\big(\{m_i+(\RR_+)^2\mid m_i\in\Sigma_i\}\big)$. L'ensemble $S_i$ des sommets de $\rN_i$ est contenu dans~$\Sigma_i$. Il est clair que $\rN_1+\rN_2$ est l'enveloppe convexe des quadrants $(m_1+m_2)+(\RR_+)^2$, avec $m_i\in\Sigma_i$, et l'ensemble $S$ des sommets est donc contenu dans $\Sigma_1+\Sigma_2$.

Soit $m\in\Sigma_1+\Sigma_2$. Supposons que $m=m_1+m_2=m'_1+m'_2$ avec $m_i,m'_i\in\Sigma_i$. Posons $n=m_1-m'_1=m'_2-m_2$. Alors, pour tous $\lambda_i\in[0,1]$ ($i=1,2$), $\lambda_im_i+(1-\lambda_i)m'_i\in \rN_i$, donc $m+(\lambda_1-\lambda_2)n\in(\rN_1+\rN_2)$. Si $m'_i\neq m_i$, alors $n\neq0$ et $m$ est intérieur à un segment entièrement contenu dans $\rN$, donc $m\not\in S$.
\end{proof}

\'Etant donnés deux polygones de Newton décorés $\wt \rN_1$ et $\wt \rN_2$, il y a alors un sens à parler de la somme de Minkowski $\wt \rN_1+\wt \rN_2$ des polygones décorés: le polygone de Newton sous-jacent est $\rN_1+\rN_2$, et chaque sommet $m$ s'écrivant de manière unique comme somme d'un sommet $m_1$ de $\rN_1$ et d'un sommet $m_2$ de $\rN_2$, on peut décorer $m$ du produit des monômes correspondant à $m_1$ et $m_2$.

\subsection{Thom-Sebastiani et transformation de Fourier}\label{subsec:1}
La relation du théorème est plus simple à comprendre après transformation de Fourier entre le plan de la variable complexe $s$ et celui de la variable complexe $\sigma$. Notons $\wh T_{f,\varphi}$ la transformée de Fourier de $T_{f,\varphi}$: pour une fonction $\eta(\sigma)$ dans la classe de Schwartz (ou $C^\infty$ à support compact), on a
\begin{align*}
\big\langle\wh T_{f,\varphi},\eta\big\rangle&=\big\langle T_{f,\varphi},\wh\eta\big\rangle\\
&=\Big\langle T_{f,\varphi},\int_\CC e^{\ov{s\sigma}-s\sigma}\eta(\sigma)\itwopi\, d\sigma\wedge d\ov\sigma\Big\rangle\\
&=\int_{\CC^m\times\CC}e^{\ov{f\sigma}-f\sigma}\eta(\sigma)\itwopi\varphi\wedge d\sigma\wedge d\ov\sigma.
\end{align*}
Puisque $T_{f,\varphi}$ est à support compact, $\wh T_{f,\varphi}$ est une fonction $C^\infty$, à croissance modérée en $\sigma=\infty$, et
\[
\big\langle\wh T_{f,\varphi},\eta\big\rangle=\int\wh T_{f,\varphi}(\sigma)\eta(\sigma)\itwopi\, d\sigma\wedge d\ov\sigma\quad\text{avec }
\wh T_{f,\varphi}(\sigma)=\int_{\CC^m} e^{\ov{f\sigma}-f\sigma}\varphi.
\]

Il est clair que
\begin{equation}\label{eq:tsfourier}
\wh T_{f\oplus g,\varphi\wedge\psi}(\sigma)=\wh T_{f,\varphi}(\sigma)\cdot \wh T_{g,\psi}(\sigma).
\end{equation}
Cette remarque est à la base de beaucoup de démonstrations de formules de type Thom-Sebastiani. En définissant de manière naturelle (\cf \S\ref{subsec:3}) un polygone de Newton décoré $\wt\rN(\wh T_{f,\varphi})$, la formule \eqref{eq:tsfourier} implique de manière évidente
\begin{equation}\label{eq:tsnewton}
\wt \rN(\wh T_{f\oplus g,\varphi\wedge\psi})=\wt \rN(\wh T_{f,\varphi})+\wt \rN(\wh T_{g,\psi}).
\end{equation}

Aussi le théorème de type \og Thom-Sebastiani\fg est une conséquence immédiate de la comparaison faite à la proposition \ref{prop:NN} entre $\wt\rN(\wh T)$ et $\wh\rN(T/C^\infty)$ pour certaines distributions $T$ holonomes régulières telles que $T_{f,\varphi},T_{g,\psi}$.\qed

\subsection{Localisation et transformation de Fourier locale}\label{subsec:2}
Notons $\CRH$ l'espace des germes en $s=0$ de distributions qui sont combinaisons linéaires finies à coefficients $C^\infty$ de distributions holonomes régulières (\cf \cite{Kashiwara86}, voir aussi \cite{Bjork93}). Ce sont aussi les combinaisons linéaires finies à coefficients $C^\infty$ des distributions $|s|^{2r}(\log|s|)^\ell$, avec $\reel(r)\in{}]-1,0]$ et $\ell\in\NN$, et de leurs dérivées (holomorphes et anti-holomorphes), comme par exemple la distribution de Dirac $\delta=\itwopi\partial_s\ov\partial_s\log|s|^2$. Dans la suite, nous nous restreindrons au sous-espace (noté de la même manière) pour lequel les $r$ possibles sont réels (on a même $r\in\QQ$ pour $T_{f,\varphi}$).

Pour un tel germe $T$, nous sous-entendrons en général le choix d'un représentant, que nous verrons comme une distribution à support compact sur le plan de la variable~$s$ par multiplication par une fonction $C^\infty$ à support compact $\rho(s)$, identiquement égale à $1$ au voisinage de $s=0$ où $T$ est définie. Nous noterons alors $\rho T$ cette distribution à support compact.

Notons $\CRHm$ le quotient de $\CRH$ par l'espace des germes de distributions à support l'origine. C'est l'espace des combinaisons linéaires finies à coefficients $C^\infty$ des $|s|^{2(r-\nu)}(\log|s|)^\ell$, avec~$r,\ell$ comme ci-dessus et $\nu\in\NN$. Pour $T\in\CRH$, notons $\wt T$ son image dans $\CRHm$.

Si $S$ est dans $\CRH$, la transformée de Fourier $\wh{\rho S}$ est une fonction $C^\infty$ de la variable $\sigma$, à croissance modérée lorsque $\sigma\to\infty$. Si on modifie $\rho$, on perturbe $\wh{\rho S}$ par une fonction de la classe de Schwartz en $\sigma$, c'est-à-dire asymptotiquement plate à l'infini. De plus, on s'intéressera uniquement au germe de $\wh{\rho S}$ en $\sigma=\infty$. Ceci revient à se permettre de multiplier $\wh{\rho S}$ par une fonction $\wh\varrho(\sigma)$ qui est nulle sur un grand disque et identiquement égale à $1$ au voisinage de $\sigma=\infty$. Par transformation de Fourier inverse, on se permet de perturber $\rho S$ par une fonction de la classe de Schwartz en $s$. On peut ainsi perturber~$S$ par n'importe quel germe de fonction $C^\infty$.

Ici, nous considérons les distributions comme des fonctions généralisées, qui s'accouplent à des fonctions $C^\infty$ à support compact ou de la classe de Schwartz via le choix des formes volume $\itwopi\,ds\wedge d\ov s$ et $\itwopi\,d\sigma\wedge d\ov\sigma$. Alors, la transformation de Fourier $T\mto \wh T$ des distributions tempérées sur~$\CC$, de noyau $e^{\ov{s\sigma}-s\sigma}$ induit un isomorphisme (transformation de Fourier locale $(0,\infty)$)
\[
\left\{
\begin{array}{l}
\text{germes de distrib. en $s=0$}\\
\text{modulo fonct. $C^\infty$ en $s=0$}
\end{array}
\right\}
\longleftrightarrow
\left\{
\begin{array}{l}
\text{germes de fonct. modérées en $\sigma=\infty$}\\
\text{modulo fonct. infiniment plates en $\sigma=\infty$}\\
\end{array}
\right\}
\]
Elle a pour inverse la transformation de Fourier locale $(\infty,0)$ de noyau $e^{s\sigma-\ov{s\sigma}}$. Autrement dit, en posant $\tau=1/\sigma$ et en notant la variable en indice, on a une correspondance bijective
\[
\CRH_s/C^\infty_s\longleftrightarrow\CRH^\rmod_\tau/C^{\infty,\text{plates}}_\tau.
\]

\subsection{Polygone de Newton décoré}
Soit $T\in\CRH_s$. Supposons que $T$ s'écrive comme une somme finie
\begin{equation}\label{eq:dvpt}
T=\sum_{r\in{}]-1,0]}\sum_{k\geq0}\theta_{r,k}(s)|s|^{2r}(\log|s|)^k
\end{equation}
avec $\theta_{r,k}(s)$ $C^\infty$ au voisinage de $s=0$. Considérons $T$ modulo les germes de fonctions~$C^\infty$. On définit alors le polygone de Newton $\rN(T/C^\infty)$ comme l'enveloppe convexe des quadrants (avec $(m',m'')\in\NN^2$)
\begin{itemize}
\item
$(m'+r,m''+r)+(\RR_+)^2$ pour lesquels il existe $k\geq0$ tel que $(\partial_s^{m'}\ov\partial{}_s^{m''}\theta_{r,k})(0)\neq0$, si $r\neq0$,
\item
$(m',m'')+(\RR_+)^2$ pour lesquels il existe $k\geq0$ tel que $(\partial_s^{m'}\ov\partial{}_s^{m''}\theta_{0,k+1})(0)\neq0$, si $r=0$ (ce décalage est dû à la négligence des fonctions $C^\infty$).
\end{itemize}

En prenant le développement de Taylor des $\theta_{r,k}$, on obtient
\[
T=\sum_{r\in{}]-1,0]}\sum_{m',m''\in\NN} P_{(m'+r,m''+r)}(\log|s|)\cdot|s|^{2r}s^{m'}\ov s{}^{m''},
\]
où les $P_{(m'+r,m''+r)}$ sont des polynômes dans $\CC[u]$, et on peut raffiner la donnée du polygone $\rN(T/C^\infty)$. Appelons \emph{polygone de Newton décoré} $\wt\rN(T/C^\infty)$ de $T$ la donnée du polygone de Newton $\rN(T/C^\infty)$ et,
\begin{itemize}
\item
pour chaque sommet $(m'+r,m''+r)$ non entier, la donnée du monôme dominant $\bC_{r,m',m''}(T)u^\ell$ de $P_{(m'+r,m''+r)}$,
\item
pour chaque sommet entier $(m',m'')$, la donnée de $\bC_{0,m',m''}(T)u^{\ell-1}$, monôme dominant de $P_{(m',m'')}$ divisé par $u$ (par définition, pour un tel sommet, on a $\ell\geq1$).
\end{itemize}

Une autre décoration sera plus utile, et nous noterons $\wh\rN(T/C^\infty)$ le polygone décoré correspondant. Nous noterons de même $\wh\bC_{r,m',m''}(T)u^\ell$ (\resp $\wh\bC_{0,m',m''}(T)u^{\ell-1}$) les monômes associés au sommet $(m'+r,m''+r)$ (\resp $(m',m'')$) de $\rN(T/C^\infty)$, avec (pour $r\in{}]-1,0[$)
\begin{align}\label{eq:bCneq0}
\wh\bC_{r,m',m''}&=(-1)^{m''}\frac1\pi\, \Gamma(r+m'+1)\Gamma(r+m''+1)\sin(\pi r)\,\bC_{r,m',m''}\\
\label{eq:bC0}\wh\bC_{0,m',m''}&=(-1)^{m''+1}\,\frac\ell2\,m'!\,m''!\,\bC_{0,m',m''}
\end{align}

\subsection{Polygone de Newton et de Newton-Mellin décorés d'une distribution modérée}\label{subsec:3}
Notons $t$ une coordonnée sur $\CC$ (qui pourra être $s$ ou $\tau=1/\sigma$ suivant les cas) et considérons un germe de distribution modérée $\wT\in\CRH^\rmod_t$ (qui pourra être $\wt T$ ou $\wh T$ suivant les cas), qui est par définition une combinaison linéaire finie à coefficients $C^\infty$ en~$t$ de fonctions modérées $|t|^{2\alpha}(\log|t|)^\ell$. Pour une telle distribution modérée $\wT$, en développant
\[
\wT(t)=\sum_{r\in{}]-1,0]}\sum_{m',m''\in\ZZ} P_{(m'+r,m''+r)}(\log|t|)\cdot|t|^{2r}t^{m'}\ov t{}^{m''},
\]
où les $P_{(m'+r,m''+r)}$ sont des polynômes dans $\CC[u]$, j'associe le polyèdre $\rN(\wT)$ de $\RR^2$ enveloppe convexe des quadrants $(m'+r,m''+r)+(\RR_+)^2$, pour lesquels $P_{(m'+r,m''+r)}\neq\nobreak0$ dans $\CC[u]$. Ceci ne dépend que du germe de~$\wT$ modulo les fonctions infiniment plates.

Appelons \emph{polygone de Newton décoré} $\wt \rN(\wT)$ de $\wT$ la donnée du polygone de Newton $\rN(\wT)$ et, pour chaque sommet $(m'+r,m''+r)$, du monôme dominant de $P_{(m'+r,m''+r)}$ (c'est-à-dire le couple formé du degré $\ell$ de $P$ et du coefficient dominant).

\begin{remarque*}
Pour $T\in\CRH$, de distribution modérée associée $\wt T\in\CRHm$, les polygones $\rN(T/C^\infty)$ et $\rN(\wt T)$ ne coïncident pas nécessairement.
\end{remarque*}

Nous allons retrouver, pour $\wT$ comme ci-dessus, le polygone de Newton décoré $\wt \rN(\wT)$ de $\wT$ à partir de données des transformées de Mellin de $\wT$ (c'est-à-dire les fonctions données par les formules \eqref{eq:mellin} ci-dessous). Nous allons donc définir un polygone de Newton-Mellin décoré de $\wT$ et montrer qu'il coïncide avec $\wt \rN(\wT)$. Commençons par définir les coefficients de Mellin de $\wT$.

Soit $\wT\in\CRH^\rmod_t$. Pour tous $m',m''\in\ZZ$, choisissons $\nu\in\NN$ de sorte que $k':=m'+\nu\in\NN$ et $k'':=m''+\nu\in\NN$. Soit $\rho(t)$ une fonction de troncature comme au \S\ref{subsec:2}. Considérons la fonction
\begin{equation}\label{eq:mellin}
\cI_{\wT}^{(k',k'')}(\lambda)=\int\wT\cdot |t|^{2\lambda}t^{-k'}\ov t{}^{-k''}\rho(t)\itwopi\,dt\wedge d\ov t.
\end{equation}
Cette fonction est holomorphe en $\lambda$ pour $\lambda\gg0$ et il est classique qu'elle se prolonge en une fonction méromorphe sur le plan complexe de la variable~$\lambda$. Pour tout $r\in{}]-1,0]$ et tout $k\in\NN$, notons $c_{r,m',m'',k}(\wT)$ (appelé \emph{coefficient de Mellin}) le coefficient de $(\lambda+r-\nu+1)^{-(k+1)}$. Si on change le choix de $\rho$, on ajoute à cette fonction une fonction entière en $\lambda$, donc $c_{r,m',m'',k}(\wT)$ ne dépend pas de $\rho$ (nous utiliserons $\rho$ et~$\rho^2$ dans la démonstration de la proposition \ref{prop:pp}). Si on change $\nu$ en $\nu+1$, alors on change $k',k''$ en $k'+1,k''+1$, et on a $\cI_{\wT}^{(k'+1,k''+1)}(\lambda)=\cI_{\wT}^{(k',k'')}(\lambda-1)$, et de même on ne change pas $c_{r,m',m'',k}(\wt T)$.

Posons alors
\begin{equation}\label{eq:Cc}
C_{r,m',m'',k}(\wT)=\frac{(-2)^k}{k!}c_{r,m',m'',k}(\wT).
\end{equation}

\begin{definition*}\mbox{}
\begin{enumeratea}
\item
Le polygone de Newton-Mellin $\rNM(\wT)$ est l'enveloppe convexe des quadrants $(m'+r,m''+r)+(\RR_+)^2$ pour lesquels il existe $k\geq0$ avec $C_{r,m',m'',k}(\wT)\neq0$.
\item
Le polygone de Newton-Mellin décoré $\wt{\rNM}(\wT)$ est obtenu en décorant chaque sommet $(m'+r,m''+r)$ de $\rNM(\wT)$ par le monôme $\bC_{r,m',m''}(\wT)\cdot u^\ell$, où $\ell$ est le plus grand entier $k\geq0$ tel que $C_{r,m',m'',k}(\wT)\neq0$ et $\bC_{r,m',m''}(\wT)=C_{r,m',m'',\ell}(\wT)$.
\end{enumeratea}
\end{definition*}

\begin{lemme*}
Les polygones décorés de Newton et de Newton-Mellin associés à une fonction $\wT\in\CRH^\rmod_t$ sont égaux.
\end{lemme*}

\begin{proof}
Il suffit de le montrer pour la fonction $\wT(t)=|t|^{2r}t^{m'}\ov t{}^{m''}(\log|t|)^\ell$. Alors, pour $k',k''\in\NN$, $\cI_{\wT}^{(k',k'')}(\lambda)$ a un pôle si et seulement si $m'-k'=m''-k''=:-\nu$. Dans ce cas, le pôle est en $\lambda_o=\nu-r-1$, il est d'ordre $\ell+1$ et le coefficient dominant est $(-1)^\ell\ell!/2^\ell$.
\end{proof}

\subsection{Coefficients de Mellin et transformation de Fourier locale}\label{subsec:coefMellin}
À $T\in\CRH_s$ on associe d'une part $\wt T\in \CRH^\rmod_s$ et d'autre part $\wh T\in\CRH^\rmod_\tau$ (plus exactement~$\wh{\rho T}$). La proposition suivante compare les coefficients de Mellin de $\wt T$ et $\wh T$.

\begin{proposition}\label{prop:cc}\mbox{}
\begin{enumeratea}
\item\label{prop:cc1}
Si $r\in{}]-1,0[$, on a, avec les notations ci-dessus, pour tous $m',m''\in\ZZ$ et $k\in\NN$,
\[
c_{r,m'+1,m''+1,k}(\wh{\rho T})=(-1)^{m''}\frac1\pi\, \Gamma(r+m'+1)\Gamma(r+m''+1)\sin(\pi r)\,c_{r,m',m'',k}(\wt T).
\]
\item\label{prop:cc2}
Si $r=0$, on a, pour tous $m',m''\in\NN$ et tout $k\in\NN$,
\begin{align}
\tag*{(\protect\ref{prop:cc2})$_{\geq(1,1)}$}\label{eq:cc2a}
c_{0,m'+1,m''+1,k}(\wh{\rho T})&=(-1)^{m''}m'!\,m''!\,c_{0,m',m'',k+1}(\wt T);\\
\tag*{(\protect\ref{prop:cc2})$_{(\geq1,0)}$}\label{eq:cc2c}
c_{0,m'+1,0,k}(\wh{\rho T})&=m'!\,c_{0,m',-1,k}(\wt T);\\
\tag*{(\protect\ref{prop:cc2})$_{(0,\geq1)}$}\label{eq:cc2b}
c_{0,0,m''+1,k}(\wh{\rho T})&=(-1)^{m''+1}m''!\,c_{0,-1,m'',k}(\wt T).
\end{align}
\end{enumeratea}
\end{proposition}

Le résultat de base est la proposition ci-dessous, qui se montre à l'aide de résultats sur les fonctions de Bessel (voir l'appendice). Pour une fonction $\cI(\lambda)$ méromorphe en $\lambda_o$, notons $\cP_{\lambda_o}\cI$ sa partie polaire en~$\lambda_o$ et, pour $k\geq0$, $\cP_{\lambda_o}^k\cI$ le coefficient de $(\lambda-\lambda_o)^{-(k+1)}$.

\begin{proposition}\label{prop:pp}
Pour tout $\lambda_o\in\CC$, on a
\begin{equation}\tag*{(\protect\ref{prop:pp})($*$)}\label{eq:pp*}
\cP_{\lambda_o}\frac{1}{\Gamma(-\lambda)}\cI_{\wh{\rho T}}^{(1,1)}(\lambda) =\cP_{\lambda_o}\frac{1}{\Gamma(\lambda+1)}\cI_{\wt T}^{(0,0)}(\lambda).
\end{equation}
On a aussi
\begin{equation}\tag*{(\protect\ref{prop:pp})($**$)}\label{eq:pp**}
\begin{aligned}
\cP_{-1}\cI_{\wh{\rho T}}^{(1,0)}(\lambda)&=\cP_0\cI_{\wt T}^{(1,0)}(\lambda)\\
\cP_{-1}\cI_{\wh{\rho T}}^{(0,1)}(\lambda)&=-\cP_0\cI_{\wt T}^{(0,1)}(\lambda).
\end{aligned}
\end{equation}
\end{proposition}

\begin{corollaire}\label{cor:pp}
Pour tous $k',k''\in\NN$ et tout $\lambda_o\in\CC$ on a
\begin{equation}\tag*{(\protect\ref{cor:pp})($*$)}\label{eq:ppk*}
\cP_{\lambda_o}\frac{1}{\Gamma(-\lambda)}\cI^{(k'+1,k''+1)}_{\wh{\rho T}}(\lambda)=\cP_{\lambda_o}\frac{(-1)^{k'} \Gamma(\lambda+1)}{\Gamma(\lambda+1-k')\Gamma(\lambda+1-k'')}\,\cI^{(k',k'')}_{\wt T}(\lambda).
\end{equation}
On a aussi, pour tous $m',m''\geq0$,
\begin{equation}\tag*{(\protect\ref{cor:pp})($**$)}\label{eq:ppk**}
\begin{aligned}
\cP_{-1}\cI_{\wh{\rho T}}^{(m'+1,0)}(\lambda)&=m'!\,\cP_0\cI_{\wt T}^{(m'+1,0)}(\lambda)\\
\cP_{-1}\cI_{\wh{\rho T}}^{(0,m''+1)}(\lambda)&=(-1)^{m''+1}m''!\,\cP_0\cI_{\wt T}^{(0,m''+1)}(\lambda).
\end{aligned}
\end{equation}
\end{corollaire}

\begin{proof}[Démonstration du corollaire \ref{cor:pp}]
Pour la transformation de Fourier considérée, on~a
\begin{align*}
\big\langle \partial_s(\rho T),\wh\eta\big\rangle=-\big\langle \rho T,\partial_s\wh\eta\big\rangle=\big\langle \rho T,\wh{\sigma\eta}\big\rangle,\quad\text{\ie. }\sigma\wh{\rho T}&=\wh{\partial_s(\rho T)},\\
\big\langle \ov\partial_s(\rho T),\wh\eta\big\rangle=-\big\langle \rho T,\ov\partial_s\wh\eta\big\rangle=-\big\langle \rho T,\wh{\ov\sigma\eta}\big\rangle,\quad\text{\ie. }\ov\sigma\wh{\rho T}&=-\wh{\ov\partial_s'\rho T)},
\end{align*}
et de même
\[
\partial_\sigma\wh{\rho T}=-\wh{s\rho T},\quad\ov\partial_\sigma\wh{\rho T}=\wh{\ov s\rho T}.
\]
On en déduit alors, pour $k',k''\in\NN$, en notant $\cbbullet$ à la place des formes volumes adéquates, pour simplifier l'écriture, et en considérant aussi une fonction de troncature $\wh\varrho(\sigma)$ comme au \S\ref{subsec:2}:
\begin{align*}
\cP_{\lambda_o}\frac{1}{\Gamma(-\lambda)}&\cI^{(k'+1,k''+1)}_{\wh{\rho T}}(\lambda)=\cP_{\lambda_o}\frac{1}{\Gamma(-\lambda)}\int\wh{\rho T}\cdot|\tau|^{2(\lambda-1)}\tau^{-k'}\ov\tau{}^{-k''}\wh\varrho(1/\tau)\cbbullet\\
&=\cP_{\lambda_o}\frac{1}{\Gamma(-\lambda)}\int\wh{\rho T}\cdot|\sigma|^{-2(\lambda+1)}\sigma^{k'}\ov\sigma{}^{k''}\wh\varrho(\sigma)\cbbullet\\
&=\cP_{\lambda_o}\frac{(-1)^{k''}}{\Gamma(-\lambda)}\int(\partial_s^{k'}\ov\partial{}_s^{k''} \rho T)^\wedge \cdot|\sigma|^{-2(\lambda+1)}\wh\varrho(\sigma)\cbbullet\\
&=\cP_{\lambda_o}\frac{(-1)^{k''}}{\Gamma(\lambda+1)}\int(\partial_s^{k'}\ov\partial{}_s^{k''} \wt T)\cdot|s|^{2\lambda}\rho(s)\cbbullet\qquad\text{d'après \ref{eq:pp*}}\\
&=\cP_{\lambda_o}\frac{(-1)^{k'}\Gamma(\lambda+1)}{\Gamma(\lambda+1-k')\Gamma(\lambda+1-k'')}\int\wt T\cdot|s|^{2\lambda}s^{-k'}\ov s{}^{-k''}\rho(s)\cbbullet.
\end{align*}
De même, pour $m''\geq0$, on a
\begin{align*}
\cP_{-1}\cI_{\wh{\rho T}}^{(0,m''+1)}(\lambda)&=\cP_{-1}\int\wh{\rho T}\cdot|\tau|^{2\lambda}\ov\tau{}^{-(m''+1)}\wh\varrho(\tau)\cbbullet\\
&=(-1)^{m''}\cP_{-1}\int(\ov\partial{}_s^{m''}\rho T)^\wedge\cdot|\tau|^{2\lambda}\ov\tau{}^{-1}\wh\varrho(\tau)\cbbullet\\
&=(-1)^{m''+1}\cP_{-1}\int\ov\partial{}_s^{m''}\wt T\cdot|s|^{2\lambda}s\rho(s)\cbbullet\\
&=-\cP_{-1}\Big[\lambda\cdots(\lambda-m''+1)\int\wt T\cdot|s|^{2\lambda}s\ov s{}^{-m''}\rho(s)\cbbullet\Big]\\
&=(-1)^{m''+1}m''!\,\cP_0\cI_{\wt T}^{(0,m''+1)}(\lambda).
\end{align*}
Le calcul de $\cP_{-1}\cI_{\wh{\rho T}}^{(m'+1,0)}(\lambda)$ est analogue.
\end{proof}

\begin{proof}[Démonstration de la proposition \ref{prop:cc}]
Si $r\neq0$, les facteurs $\Gamma$ dans \ref{eq:ppk*} n'ont ni zéro ni pôle en $-r+\nu-1$, donc
\begin{align*}
c_{r,m'+1,m''+1,k}(\wh{\rho T})&=\cP_{-r+\nu-1}^k\cI^{(k'+1,k''+1)}_{\wh{\rho T}}(\lambda)\\
&=\frac{(-1)^{k'}\Gamma(r+1-\nu)\Gamma(\nu-r)}{\Gamma(\nu-r-k')\Gamma(\nu-r-k'')}\cP_{-r+\nu-1}^k\cI^{(k',k'')}_{\wt T}(\lambda)\\
&=(-1)^{m''}\frac1\pi\, \Gamma(r+m'+1)\Gamma(r+m''+1)\sin(\pi r)\,c_{r,m',m'',k}(\wt T).
\end{align*}

Supposons maintenant $r=0$ et $m',m''\in\NN$. On peut alors choisir $k'=m'$, $k''=m''$ et $\nu=0$. \ref{eq:ppk*} devient
\[
\cP_{-1}\cI^{(m'+1,m''+1)}_{\wh{\rho T}}(\lambda)=(-1)^{m''}m'!\,m''!\cP_{-1}\frac{1}{\Gamma(\lambda+1)}\cI^{(m',m'')}_{\wt T}(\lambda).
\]

Si $r=0$, $m'=-1$ et $m''\geq0$, \ref{eq:ppk**} donne, pour $k\geq0$,
\begin{align*}
c_{0,0,m''+1,k}(\wh{\rho T})&=\cP^k_{-1}\cI_{\wh{\rho T}}^{(0,m''+1)}(\lambda)\\
&=(-1)^{m''+1}m''!\,\cP^k_0\cI_{\wt T}^{(0,m''+1)}(\lambda)\\
&=(-1)^{m''+1}m''!\,c_{0,-1,m'',k}(\wt T).
\end{align*}
L'égalité pour $c_{0,m'+1,0,k}(\wh{\rho T})$ se montre de même.
\end{proof}

\Subsection{Polygone de Newton décoré et transformation de Fourier locale}\label{subsec:preuve}

\begin{proposition}\label{prop:NN}
Pour $T$ de la forme \eqref{eq:dvpt}, on a l'égalité des polygones décorés
\[
\wt\rN(\wh{\rho T})=\wh\rN(T/C^\infty)+(1,1).
\]
\end{proposition}

\begin{proof}
Vu la forme de son développement, $T$ est $L^1$ au voisinage de $s=0$, et donc $\wh{\rho T}$ est continue au voisinage de $\infty$ et tend vers $0$ quand $\sigma\to\infty$, d'après Riemann-Lebesgue. On en déduit que
\[
C_{r,m',m'',k}(\wh{\rho T})\neq0\implique
\begin{cases}
m',m''\geq1&\text{si $r\in{}]-1,0[$},\\
m',m''\geq0\text{ et non tous deux nuls}&\text{si $r=0$}.
\end{cases}
\]
Pour $r\neq0$, un point $(r+m',r+m'')+(1,1)$ est donc dans $\rN(\wh{\rho T})$ si et seulement si $m',m''\geq0$ et $C_{r,m'+1,m''+1,k}(\wh{\rho T})\neq0$. La formule \eqref{prop:cc1} de la proposition~\ref{prop:cc}, jointe à~\eqref{eq:Cc}, dit que ceci est équivalent à $m',m''\geq0$ et $C_{r,m',m'',k}(\wt T)\neq0$. Enfin, puisque $r\neq0$, ceci est équivalent à $C_{r,m',m'',k}(T)\neq0$. En considérant les monômes dominants en ce point, on obtient \eqref{eq:bCneq0} en ce point.

Supposons $r=0$. Alors, d'après les relations \ref{eq:cc2c} et \ref{eq:cc2b} de la proposition~\ref{prop:cc}, jointes à \eqref{eq:Cc} et \eqref{eq:dvpt}, on a $C_{0,m',0,k}(\wh{\rho T})=0$ et $C_{0,0,m'',k}(\wh{\rho T})=0$ pour tous $k\geq0$ et $m'\geq 1$ ou $m''\geq1$. Par suite, $C_{0,m',m'',k}(\wh{\rho T})\neq0$ implique $m',m''\geq1$, et la formule \ref{eq:cc2a} de la proposition~\ref{prop:cc}, jointe à \eqref{eq:Cc}, montre que $C_{0,m'+1,m''+1,k}(\wh{\rho T})\neq0$ si et seulement si $m',m''\geq0$ et $C_{r,m',m'',k+1}(\wt T)\neq0$, et donc, puisque $k\geq0$, si et seulement si $C_{r,m',m'',k+1}(T)\neq0$. On a alors, pour tout $k\geq0$,
\begin{align*}
C_{0,m',m'',k}(\wh{\rho T})&=\frac{(-2)^k}{k!}c_{0,m',m'',k}(\wh{\rho T})\\
&=\frac{(-2)^k}{k!}(-1)^{m''}m'!m''!\,\frac{(k+1)!}{(-2)^{k+1}}\,C_{r,m',m'',k+1}(\wt T)\\
&=(-1)^{m''+1}\,\frac{(k+1)}{2}\,m'!m''!C_{r,m',m'',k+1}(T).
\end{align*}
Par conséquent, $(m',m'')+(1,1)\in\rN(\wh{\rho T})$ si et seulement si $(m',m'')\in\rN(T/C^\infty)$. On a de plus en ce point la relation \eqref{eq:bC0}.

En conclusion, les relations \eqref{eq:bCneq0} et \eqref{eq:bC0} valent en tous les points du support des développements de $\wh{\rho T}$ et $T$ modulo $C^\infty$, donc les polygones décorés coïncident.
\end{proof}

\subsection*{Appendice: démonstration de la proposition \ref{prop:pp}}
C'est la même que celles de \cite[Prop\ptbl5.8]{Bibi05b} et de \cite[Lemme 3.6.21]{Bibi01c}, en supprimant les considérations inutiles ici. Notons, pour $\rho(s)$ $C^\infty$ à support compact et $\equiv1$ près de $s=0$, et pour tous $\ell',\ell''\in\ZZ$,
\[
\wh I_\rho^{(\ell',\ell'')}(\sigma,\lambda)=\int e^{s\sigma-\ov{s\sigma}}|s|^{2\lambda}s^{\ell'}\ov s^{\ell''}\rho(s)\itwopi\,ds\wedge d\ov s
\]
et $\wh I_\rho(\sigma,\lambda)=\wh I_\rho^{(0,0)}(\sigma,\lambda)$ pour simplifier. On a les propriétés suivantes:
\begin{enumeratea}
\item
Si $\reel(\lambda+1+(\ell'+\ell'')/2)>0$, alors $(\sigma,\lambda)\mto\wh I_\rho^{(\ell',\ell'')}(\sigma,\lambda)$ est $C^\infty$ et holomorphe en $\lambda$ et $\lim_{\sigma\to\infty}\wh I_\rho^{(\ell',\ell'')}(\sigma,\lambda)=0$ localement uniformément par rapport à $\lambda$. C'est la transformée de Fourier (inverse, si on prend la définition du \S\ref{subsec:1}) de la fonction $s\mto|s|^{2\lambda}s^{\ell'}\ov s^{\ell''}\rho(s)$.
\item
On a, sur le domaine où les termes d'une égalité sont $C^\infty$,
\begin{align*}
-\sigma\wh I_\rho^{(\ell',\ell'')}&=(\lambda+\ell')\wh I_\rho^{(\ell'-1,\ell'')}+\wh I_{\partial_s\rho}^{(\ell',\ell'')}&\partial_\sigma\wh I_\rho^{(\ell',\ell'')}&=\wh I_\rho^{(\ell'+1,\ell'')}\\
\ov\sigma\wh I_\rho^{(\ell',\ell'')}&=(\lambda+\ell'')\wh I_\rho^{(\ell',\ell''-1)}+\wh I_{\ov\partial_s\rho}^{(\ell',\ell'')}&-\ov\partial_\sigma\wh I_\rho^{(\ell',\ell'')}&=\wh I_\rho^{(\ell',\ell''+1)}
\end{align*}
et on remarque aussi que $\wh I_{\partial_s\rho}^{(\ell',\ell'')}$ et $\wh I_{\ov\partial_s\rho}^{(\ell',\ell'')}$ sont $C^\infty$ sur $\PP^1\times\CC$ et holomorphes en $\lambda$, et infiniment plats en $\sigma=\infty$ localement uniformément par rapport à $\lambda$. En particulier, pour $\reel(\lambda+1)>0$, on a
\begin{align*}
-\sigma\partial_\sigma\wh I_\rho&=(\lambda+1)\wh I_\rho+\wh I_{\partial_s\rho}^{(1,0)}\\
-\ov\sigma\ov\partial_\sigma\wh I_\rho&=(\lambda+1)\wh I_\rho+\wh I_{\ov\partial_s\rho}^{(0,1)}.
\end{align*}
\end{enumeratea}

\begin{proof}[Démonstration de \ref{eq:pp*}]
Notons maintenant
\[
\wh J_\rho(\sigma,\lambda)=|\sigma|^{2(\lambda+1)}\wh I_\rho(\sigma,\lambda).
\]
Alors, d'après les relations ci-dessus, en se rappelant que $\tau\partial_\tau$ agit comme $-\sigma\partial_\sigma$, on~a, pour $\reel(\lambda+1)>0$,
\[
\tau\partial_\tau\wh J_\rho=\wh J_{\partial_s\rho}^{(1,0)},\quad \ov\tau\ov\partial_\tau\wh J_\rho=\wh J_{\ov\partial_s\rho}^{(0,1)},
\]
et ces deux fonctions s'étendent en des fonctions $C^\infty$ au voisinage de $\tau=0$, holomorphes par rapport à $\lambda$, et sont infiniment plates en $\tau=0$, localement uniformément par rapport à $\lambda\in\CC$.

\begin{lemme*}
Pour $\reel(\lambda+1)\in{}]0,1/4[$, on a $\lim_{\sigma\to\infty}\wh J_\rho(\sigma,\lambda)=\Gamma(\lambda+1)/\Gamma(-\lambda)$.
\end{lemme*}

\begin{proof}
On peut supposer que $\rho\equiv1$ pour $|s|\leq1$, et on est amené à calculer la limite de
\[
|\sigma|^{2(\lambda+1)}\int_{|s|\leq1}e^{s\sigma-\ov{s\sigma}}|s|^{2\lambda}\itwopi\,ds\wedge d\ov s.
\]
Cette limite s'écrit (en posant $u=s\sigma$ puis $r'=|u|$ et $\theta=\arg u$)
\[
\int_0^\infty\frac1\pi\int_0^{2\pi} e^{2ir'\sin\theta}r^{\prime2\lambda+1}d\theta\,dr'.
\]
En utilisant la fonction de Bessel $J_0(r)=\frac1{2\pi}\int_0^{2\pi}e^{-ir\sin\theta}d\theta=\frac1{2\pi}\int_0^{2\pi}e^{ir\sin\theta}d\theta$, cette limite est
\[
2\int_0^\infty r^{\prime2\lambda+1}J_0(2r')dr'=\frac1{2^{2\lambda+1}}\int_0^\infty r^{2\lambda+1}J_0(r)dr
\]
et il est connu (\cf \cite[\S13.24, p\ptbl391]{Watson22}) que, sur la bande $\reel(\lambda+1)\in{}]0,1/4[$, l'intégrale précédente vaut $2^{2\lambda+1}\Gamma(\lambda+1)/\Gamma(-\lambda)$.
\end{proof}

Considérons la fonction
\[
\wh K_\rho(\tau,\lambda)=\int_0^1\big[\wh J_{\partial_s\rho}^{(1,0)}(u\tau,\lambda)+\wh J_{\ov\partial_s\rho}^{(0,1)}(u\tau,\lambda)\big]du.
\]
Cette fonction est $C^\infty$ au voisinage de $\tau=0$, holomorphe en $\lambda\in\CC$ et infiniment plate en \hbox{$\tau=0$}, localement uniformément par rapport à $\lambda$ (puisque c'est le cas pour l'intégrant). Il en est de même de la fonction $\wh L_\rho(\tau,\lambda)=|\tau|^{2(\lambda+1)}\wh K_\rho(\tau,\lambda)$.

Sur la bande donnée dans le lemme, on peut donc écrire, au voisinage de $\tau=0$,
\[
\wh J_\rho(\sigma,\lambda)=\frac{\Gamma(\lambda+1)}{\Gamma(-\lambda)}+\wh K_\rho(\tau,\lambda),
\]
et donc $\wh I_\rho(\sigma,\lambda)=|\sigma|^{-2(\lambda+1)}\frac{\Gamma(\lambda+1)}{\Gamma(-\lambda)}+\wh L_\rho(\tau,\lambda)$.
Puisque, pour $\tau\neq0$ fixé, ces deux fonctions sont holomorphes sur $\reel(\lambda+1)>0$ et coïncident sur la bande $\reel(\lambda+1)\in{}]0,1/4[$, elles coïncident sur $\reel(\lambda+1)>0$.

Revenons à la proposition. Le terme de gauche de \ref{eq:pp*} s'écrit encore
\[
\cP_{\lambda_o}\frac{1}{\Gamma(-\lambda)}\int\wh{\rho T}\cdot|\sigma|^{-2(\lambda+1)}\wh\varrho(\sigma)\itwopi\,d\sigma\wedge d\ov\sigma.
\]
On voit alors que
\[
\int\wh{\rho T}\wh\varrho(\sigma)\wh L_\rho(\sigma,\lambda)\itwopi\,d\sigma\wedge d\ov\sigma
\]
est une fonction entière de $\lambda$. Par ailleurs, on a
\begin{equation}\label{eq:TI}
\begin{aligned}
\int\wt T\cdot|s|^{2\lambda}\rho(s)^2\itwopi\,ds\wedge d\ov s&=\big\langle\rho T,|s|^{2\lambda}\rho(s)\itwopi\,ds\wedge d\ov s\big\rangle\\
&=\int\wh{\rho T}\cdot\wh I_\rho(\sigma,\lambda)\itwopi\,d\sigma\wedge d\ov\sigma.
\end{aligned}
\end{equation}
Pour $\wh\varrho(\sigma)\equiv1$ près de $\sigma=\infty$ et $\equiv0$ près de $\sigma=0$, $(1-\wh\varrho(\sigma))\cdot\wh{\rho T}$ est à support compact, sa transformé de Fourier inverse $\eta$ est dans la classe de Schwartz en~$s$, et~\eqref{eq:TI} lue en sens inverse appliquée à $(1-\wh\varrho(\sigma))\cdot\wh{\rho T}$ donne
\begin{equation}\label{eq:varrhohat}
\int(1-\wh\varrho(\sigma))\cdot\wh{\rho T}\cdot\wh I_\rho(\sigma,\lambda)\itwopi\,d\sigma\wedge d\ov\sigma=\int\eta\rho(s)|s|^{2\lambda}\itwopi ds\wedge d\ov s=\Gamma(\lambda+1)\cdot h(\lambda)
\end{equation}
où $h$ est entière par rapport à $\lambda$. On obtient finalement l'égalité de fonctions méromorphes sur $\CC$:
\begin{multline}\label{eq:TT}
\int\wt T\cdot|s|^{2\lambda}\rho(s)^2\itwopi\,ds\wedge d\ov s
=\frac{\Gamma(\lambda+1)}{\Gamma(-\lambda)}\int\wh{\rho T}\cdot|\sigma|^{-2(\lambda+1)}\wh\varrho(\sigma)\itwopi\,d\sigma\wedge d\ov\sigma\\
{}+\int\wh{\rho T}\wh\varrho(\sigma)\wh L_\rho(\sigma,\lambda)\itwopi\,d\sigma\wedge d\ov\sigma+\Gamma(\lambda+1)h(\lambda),
\end{multline}
et on conclut en utilisant le fait que le deuxième terme du membre de droite est une fonction entière de $\lambda$, de même que le troisième divisé par $\Gamma(\lambda+1)$.
\end{proof}

\begin{proof}[Démonstration de \ref{eq:pp**}]
Considérons maintenant $\wh I_\rho^{(1,0)}$ et $\wh J_\rho^{(1,0)}=\sigma|\sigma|^{2(\lambda+1)}\wh I_\rho^{(1,0)}$. On a une propriété analogue à celle de $\wh J_\rho$ pour les dérivées $\tau\partial_\tau\wh J_\rho^{(1,0)}$ et $\ov\tau\ov\partial_\tau\wh J_\rho^{(1,0)}$.

\begin{lemme*}
Pour $\reel(\lambda+1)\in{}]-1,-1/4[$, on a $\lim_{\sigma\to\infty}\wh J_\rho^{(1,0)}(\sigma,\lambda)=-\Gamma(\lambda+2)/\Gamma(-\lambda)$.
\end{lemme*}

\begin{proof}
Comme pour le cas \ref{eq:pp*}, on se ramène à calculer la limite de
\[
\sigma|\sigma|^{2(\lambda+1)}\int_{|s|\leq1}e^{s\sigma-\ov{s\sigma}}s|s|^{2\lambda}\itwopi\,ds\wedge d\ov s.
\]
En utilisant la fonction de Bessel $J_{\pm1}(r)=\frac1{2\pi}\int_0^{2\pi} e^{-ir\sin\theta}e^{\pm i\theta}d\theta$, cette limite s'écrit
\begin{equation}\label{eq:J}
2\int_0^\infty r^{\prime2(\lambda+1)}J_{-1}(2r')dr'=\frac{-1}{2^{2(\lambda+1)}}\int_0^\infty r^{2(\lambda+1)}J_1(r)dr,
\end{equation}
puisque $J_1=-J_{-1}$. Pour $\reel(\lambda+1)\in{}]-1,-1/4[$, il est connu (\cf loc\ptbl cit.) que l'intégrale précédente est égale à $2^{2(\lambda+1)}\Gamma(\lambda+2)/\Gamma(-\lambda)$.
\end{proof}

On peut donc écrire $\wh I_\rho^{(1,0)}(\sigma,\lambda)=-\sigma^{-1}|\sigma|^{-2(\lambda+1)}\frac{\Gamma(\lambda+2)}{\Gamma(-\lambda)}+\wh L_\rho^{(1,0)}(\tau,\lambda)$, où $\wh L_\rho^{(1,0)}$ a les mêmes propriétés que $\wh L_\rho$ ci-dessus. D'autre part,
\[
\cI_{\wh{\rho T}}^{(0,1)}(\lambda)=\int\wh{\rho T}\module{\tau}^{2\lambda}\ov\tau^{-1}\wh\varrho(1/\tau)\itwopi d\tau\wedge d\ov\tau=\int\wh{\rho T}\sigma^{-1}\module{\sigma}^{-2(\lambda+1)}\wh\varrho(\sigma)\itwopi d\sigma\wedge d\ov\sigma.
\]
Comme avec \eqref{eq:varrhohat}, on montre que
\[
\int\wh{\rho T}\cdot(1-\wh\varrho)\wh I_\rho^{(1,0)}(\sigma,\lambda)\itwopi\,d\sigma\wedge d\ov\sigma=\Gamma(\lambda+2)h(\lambda)
\]
avec $h(\lambda)$ entière. Comme en \eqref{eq:TI}, on trouve
\[
\int\wt T\cdot|s|^{2\lambda}s\rho(s)^2\itwopi\,ds\wedge d\ov s=\int\wh{\rho T}\cdot\wh I_\rho^{(1,0)}(\sigma,\lambda)\itwopi\,d\sigma\wedge d\ov\sigma,
\]
et \eqref{eq:TT} devient
\begin{multline*}
\frac{-1}{\Gamma(\lambda+2)}\int\wt T\cdot|s|^{2\lambda}s\rho(s)^2\itwopi\,ds\wedge d\ov s\\
=\frac{1}{\Gamma(-\lambda)}\int\wh{\rho T}\sigma^{-1}\module{\sigma}^{-2(\lambda+1)}\wh\varrho(\sigma)\itwopi d\sigma\wedge d\ov\sigma+g(\lambda)
\end{multline*}
avec $g(\lambda)$ entière, c'est-à-dire
\[
\frac{-1}{\Gamma(\lambda+2)}\cI_{\wt T}^{(0,1)}(\lambda+1)=\frac{-1}{\Gamma(\lambda+2)}\cI_{\wt T}^{(-1,0)}(\lambda)=\frac{1}{\Gamma(-\lambda)}\cI_{\wh{\rho T}}^{(0,1)}(\lambda)+g(\lambda).
\]
On utilise alors le fait que $\Gamma(1)=1$ pour conclure. Le calcul pour $\wh I_\rho^{(1,0)}$ est analogue, l'absence de signe provenant du fait qu'on obtient directement la fonction $J_1$ dans \eqref{eq:J}.
\end{proof}

\backmatter
\newcommand{\SortNoop}[1]{}\def\cprime{$'$}
\providecommand{\bysame}{\leavevmode ---\ }
\providecommand{\og}{``}
\providecommand{\fg}{''}
\providecommand{\smfandname}{\&}
\providecommand{\smfedsname}{\'eds.}
\providecommand{\smfedname}{\'ed.}
\providecommand{\smfmastersthesisname}{M\'emoire}
\providecommand{\smfphdthesisname}{Th\`ese}

\end{document}